\documentclass[12pt]{elsarticle}

\usepackage{amssymb,amsmath}
\usepackage{amsthm,color}
\usepackage[a4paper,margin=2cm]{geometry}
\theoremstyle{plain}
\newtheorem{lemma}{Lemma}

\newtheorem{theorem}{Theorem}
\theoremstyle{remark}
\newtheorem{remark}{Remark}

\DeclareMathOperator{\md}{\mathsf{D}}

\newcommand{\E}[1]{\mathsf{E}\left[\,#1\,\right]}
\newcommand{\Ex}[1]{\mathsf{E}\Big[\,#1\,\Big]}
\renewcommand{\d}[1]{\ensuremath{\operatorname{d}\!{#1}}}

\newcommand{\R}{\mathbb R}
\newcommand{\ra}{\rangle}
\newcommand{\la}{\langle}
\newcommand{\grad}{\mathop{\mathrm{grad}}\nolimits}
\newcommand{\norm}[1]{\left\lVert#1\right\rVert}
\newcommand{\abs}[1]{\left\lvert#1\right\rvert}

\newcommand*{\set}[1]{\left\{#1\right\}}
\newcommand{\ind}[1]{\mathbf{1}_{#1}}
\newcommand{\cI}{\mathcal I}

\journal{Statistics and Probability Letters}

\begin{document}

\begin{frontmatter}



\title{Malliavin regularity of solutions to mixed stochastic differential equations}

\author[label1]{Georgiy Shevchenko\corref{cor1}}
\ead{zhora@univ.kiev.ua}
\cortext[cor1]{Corresponding author}
\author[label1]{Taras Shalaiko}
\ead{tosha@univ.kiev.ua}
 \address[label1]{Department of Probability, Statistics and Actuarial Mathematics, Mechanics and Mathematics
Faculty, Taras Shevchenko National University of Kyiv, 64 Volodymyrska,  01601 Kyiv, Ukraine}


\begin{abstract}
For a mixed stochastic differential equation  driven by independent fractional Brownian motions and Wiener processes, the existence and integrability of the Malliavin derivative of the solution are established. It is also proved that the solution possesses exponential moments.
\end{abstract}

\end{frontmatter}


\section{Introduction}

This paper is devoted to the following mixed stochastic differential equation (SDE) in $\R^d$:
\begin{equation}\label{eq1}
X_t=X_0  +\int_0^ ta(X_s)\d s+ \int_0^t b(X_s)\d W_s+\int_0^t c(X_s)\d B_s,
\end{equation}
where $W=\set{W_t,t\ge 0}$ is an $m$-dimensional standard Wiener process, $B= \set{B_t, t\ge 0}$ is an $l$-dimen\-sional fractional Brownian motion; the coefficients $a:\R^d\to \R^d$, $b\colon \R^d \to \R^{d\times m}$, $c\colon \R^d \to \R^{d\times l}$ are continuous,  $X_0\in\R^d$ is non-random. (See Section~\ref{sec:prelim} for precise definitions of all objects.)

Since the seminal paper \cite{cheridito}, mixed stochastic models containing both a standard Wiener process and a fractional Brownian motion (fBm) gained a lot of attention. The main reason for this is that they allow to model systems driven by a combination of random noises, one of which is white and another has a long memory. In financial modelling, for example, one can distinguish between the randomness coming from economical situation and the randomness originating from  the market microstructure.

Unique solvability of equation \eqref{eq1} was established under different sets of conditions in \cite{GuerraNualart,Kubilius,LRDShMish,SDEMiSH}.
Article \cite{LRDShMish} also contains an important result that the solution to \eqref{eq1} is a limit of solutions to It\^o SDEs, which gives a tool to transfer some elements of the well developed theory for the It\^o SDEs to equation \eqref{eq1}.

Over the last decades, the Malliavin calculus of variations has become one of the most important tools in stochastic analysis. While originally it was developed by Malliavin to study existence and regularity of densities of solutions to SDEs, now it has numerous applications in mathematical finance, statistics, optimal control, etc. For this reason, questions of Malliavin regularity generate considerable scientific interest. There is a huge amount of articles devoted to the Malliavin calculus for It\^o SDEs, see \cite{Nual} and references therein. For SDEs driven by fBm, the questions of Malliavin regularity were studied in \cite{Baudoin,RoughDensity,MishShevTSP,Nourdin-Simon,NualartSau}. Equation \eqref{eq1} can be treated with the help of the rough path theory, which also allows to address the case $H<1/2$, see \cite{RoughDensity}. However, this requires rather high regularity of coefficients (usually they are assumed to be infinitely differentiable with all bounded derivatives). In this paper we will use another techniques, namely, the approach developed in \cite{SDEMiSH,LRDShMish,GMMixed}, to study the equation \eqref{eq1}, which enables us to prove the Malliavin regularity under less restrictive assumptions on the coefficients. One of the key ingredients of the proof is the above-mentioned approximation by the solutions of It\^o SDEs. As a side result, we prove an exponential integrability of solutions to mixed SDEs with bounded coefficients. This result is of independent interest in financial mathematics, where it can be used to prove existence of martingale measures or to argue the integrability of solutions for certain equations with stochastic volatility.

The paper is organized as follows.  Section~\ref{sec:prelim} contains necessary definitions. It also provides a brief summary on the pathwise  integration and the Malliavin calculus of variations for fractional Brownian motion. In Section~\ref{sec:integr}, we prove the main results of the article: exponential integrability of solution to \eqref{eq1} and existence and integrability of Malliavin derivatives. Proofs of auxiliary results are given in Appendix.

\section{Preliminaries}\label{sec:prelim}

\subsection{Definitions, notations and assumptions}
On a complete filtered probability space $\left(\Omega, \mathcal F, \mathbb F = (\mathcal F_t)_{t\ge 0}, \mathsf P\right)$,  let $W = \set{W_t = \left(W_t^1,\dots, W_t^m\right),t\ge 0}$ be a standard $\mathbb F$\nobreakdash-Wiener process in $\R^m$, $B = \set{B_t = \big(B_t^{H,1},\dots, B_t^{H,l}\big),t\ge 0}$ be an $\mathbb F$\nobreakdash-adapted fractional Brownian motion in $\R^l$, i.e.\ a collection of independent fBms $B^{H,k} $ with Hurst index $H\in(1/2,1)$,  independent of $W$. We recall that an fBm with Hurst index $H\in(0,1)$ is a centered Gaussian process $B^H=\{B^H_t,t\ge 0\}$ with the covariance function
$$R_H(t,s)=\frac{1}{2}(t^{2H}+s^{2H}-|t-s|^{2H});$$
in the case $H\in(1/2,1)$ considered here, $B^H$ has the property of long-range dependence. It is known that fBm has a continuous modification (even H\"older continuous of any order up to $H$), and in what follows we will assume that the process $B^H$ is continuous.

The equation \eqref{eq1}, is understood precisely as a system of equations on $[0,T]$
\begin{equation}\label{eq2}
\begin{aligned}
X^i_t=X^i_0  +\int_0^ ta_i(X_s)\d s+\sum_{j=1}^m\int_0^tb_{i,j}(X_s)\d W^j_s +\sum_{k=1}^l\int_0^tc_{i,k}(X_s)\d B^{H,k}_s,i=1,2,\dots, d,
\end{aligned}
\end{equation}
where for $i=1,\dots,d$, $j=1,\dots,m$, $k=1,\dots,l$, the functions $a_i,b_{i,j},c_{i,k}\colon \R^d\to \R$ are continuous; the integrals w.r.t.\ $W^j$ are understood in the It\^o sense, whereas those w.r.t.\ $B^{H,k}$, in the pathwise sense, as defined in \ref{subsec:ifbm}.

Throughout the paper, we will use the following notation: $\abs{\cdot}$ will denote the absolute value of a number, the Euclidean norm of a vector, and the operator norm: $\abs{Ax} = \sup_{\abs{x}=1} \abs{Ax}$. The inner product in $\R^d$ will be denoted by $\la \cdot, \cdot\ra$; $\norm{f}_\infty = \sup_{x\in\R^d} \abs{f(x)}$ is the supremum norm of a function $f$, either real-valued, vector-valued, or operator-valued. The symbol $C$ will be used for a generic constant, whose value is not important and may change from one line to another.


We will impose the following assumptions on the coefficients  $a=(a_1,\ldots,a_d)$, $b=(b_{i,j},i=1,\ldots,d,j=1,\ldots,m)$ and $c=(c_{i,j},i=1,\ldots,d,j=1,\ldots,l)$  of (\ref{eq2}):
\begin{enumerate}[({A}1)]
\item $a,b,c$ are bounded and have bounded continuous derivatives,
\item $c$ is twice differentiable and $c''$ is bounded.
\end{enumerate}

\subsection{Integration with respect to fractional Brownian motion}\label{subsec:ifbm}

The integral with respect to fBm will be understood in the pathwise (Young) sense. Specifically, for a $\nu$-H\"older continuous function $f$ and a $\mu$-H\"older continuous function $g$ with $\mu+\nu >1 $ the integral $\int_a^b f(x) dg(x)$ exists as a limit of integral sums, moreover, the inequality
\begin{gather}
\left|\int_{a}^b f(s)\d g(s)\right|\leq K_{\mu,\nu} \norm{g}_{a,b,\mu}\left(\norm{f}_{a,b,\infty}(b-a)^{\mu}+\norm{f}_{a,b,\nu}(b-a)^{\mu+\nu}\right)
\label{Estim}
\end{gather}
holds,
where $\norm{f}_{a,b,\infty}=\sup_{x\in [a,b]}|f(x)|$ is the supremum norm on $[a,b]$, and for $\gamma\in(0,1)$ $$\norm{f}_{a,b,\gamma}=\sup_{a\leq s<t\leq b}\frac{|f(t)-f(s)|}{|t-s|^{\gamma}}$$
is the H\"older seminorm on $[a,b]$; $K_{\mu,\nu}$ is a universal constant. Thus, since fBm is H\"older continuous of any order less than $H$, the integral $\int_a^b f(s) dB_s^H$ is well defined provided $f$ is $\beta$-H\"older continuous on $[a,b]$ with $\beta>1-H$.
We will use inequality \eqref{Estim} also in a multidimensional case (see e.g. \cite[Proposition 1]{NualartSau}); for simplicity we will write it with the same constant $K_{\mu,\nu}$.


\subsection{Malliavin calculus of variations}
Here we give only basics of Malliavin calculus of variations with respect to fBm, see \cite{Nual} for a deeper exposition. Let $S[0,T]$ denote the set of step functions of the form $f(t) = \sum_{k=1}^{n} c_k \ind{[a_k,b_k)}(t)$ defined on $[0,T]$. For functions $f,g\in S[0,T]$ define the scalar product   \begin{gather*}
\la f,g\ra_{H}=\int_0^T\int_0^T f(t)g(s)\phi(t,s)\d t\d s,
\end{gather*}
where $\phi(t,s)=H(2H-1)|t-s|^{2H-2}$.
Let $L^2_H[0,T]$ denote the closure of $S[0,T]$ w.r.t.\ this scalar product. It is a separable Hilbert space, which contains not only classical functions, but also some distributions (see \cite{pipiras-taqqu}).
Then the 
product $$\mathfrak H = \left(L^2_H [0,T]\right)^{ l}\times \left(L^2[0,T]\right)^{m}$$ is a separable Hilbert space with the scalar product $$\la f,g\ra_{\mathfrak H} = \sum_{i=1}^l \la f_i,g_i\ra_H + \sum_{i=l+1}^m \la f_i,g_i\ra_{L^2[0,T]}.$$
The map
$$
\cI \colon (\ind{[0,t_1)},\ind{[0,t_2)},\dots,\ind{[0,t_l)},\ind{[0,s_1)},\ind{[0,s_2)},\dots,\ind{[0,s_m)})\mapsto (B^{H,1}_{t_1},B^{H,2}_{t_2},\dots, B^{H,l}_{t_l},W^1_{s_1},W^2_{s_2},\dots, W^m_{s_m})
$$
can be extended by linearity to $S[0,T]^{l+m}$. It appears that for $f,g\in S[0,T]^{l+m}$ \ $$\E{\la \cI(f),\cI(g)\ra} = \la f,g\ra_{\mathfrak H},$$
so $\cI$ can be extended to an isometry between $\mathfrak H$ and a subspace of $L^2(\Omega;\R^{m+l})$.

For a smooth cylindrical variable of the form $\xi = F(\cI(f_1),\dots ,\cI(f_n))$, where $f_i=(f_{i,1},\ldots,f_{i,m+l})\in \mathfrak H$ for $i=1,\ldots,n$ and  $F\colon \R^{n(m+l)} \to \R$ is a continuously differentiable finitely supported function, define the
Malliavin derivative $\mathsf D \xi$ as a random element in $\mathfrak H$ with the $j$-th coordinate equal to $ \sum_{i=1}^{n} \partial_{(i-1)(l+m)+j} F (\cI(f_1),\dots, \cI(f_n)) f_{i,j}$, $j=1,\ldots,l+m$. For $p\ge 1$ denote by $\mathbb D^{1,p}$ the closure of the space of smooth cylindrical random variables with respect to the norm
$$\norm{\xi}_{\mathbb D^{1,p}}=\E{|\xi|^{p}+\norm{\md{\xi}}^p_{\mathfrak H}}^{1/p};$$
$\md$ is closable in this space and its closure will be denoted likewise. Finally, the Malliavin derivative is a (possibly, generalized) function from $[0,T]$ to $\R^{l+m}$, so we can introduce the notation
$$
\md \xi = \set{\md_t\xi = \left(\md^{H,1}_t \xi,\dots, \md^{H,l}_t\xi,\md^{W,1}_t\xi,\dots, \md^{W,m}_t\xi\right),\  t\in[0,T]}.$$

\section{Existence of exponential moments and Malliavian regularity}\label{sec:integr}
In this section we prove that certain  power of the supremum norm of the solution to (\ref{eq1}) possesses exponential moment; then this fact is used  to prove  the Malliavin differentiability of this solution.

\begin{theorem}\label{thm:integr}
The solution $X$ of \eqref{eq2} satisfies $\E{\exp\set{z\norm{X}^\alpha_{0,T,\infty}}}<\infty$ for any $\alpha\in(0,4H/(2H+1))$, $z>0$.
\end{theorem}
\begin{proof}
 The inequality $\alpha<4H/(2H+1)$ is equivalent to
$(2H)^{-1}<2\alpha^{-1} -1$, therefore, it is possible to choose some
$\nu\in(1/2,H)$ such that $(2\nu)^{-1}<2\alpha^{-1} -1$. Take also arbitrary $\beta\in \left((2\nu)^{-1},2\alpha^{-1} -1\right)$ so that $\alpha(1+\beta)<2$.

Now if  $\kappa\in(1-\nu,1/2)$  is sufficiently close to $1/2$, it follows from Lemma~\ref{lemmaX} that
$$
\norm{X}_{0,T,\infty}\le C\left(1+ \norm{B}_{0,T,\nu}^{1+\beta} + J_{X,\kappa}(T) \big(1+\norm{B}_{0,T,\nu}^{\beta}\big)\right).
$$
By the Young inequality,  $$
J_{X,\kappa}(T)\norm{B}^\beta_{0,T,\nu} \le \frac{ J_{X,\kappa}(T)^{1+\beta}}{1+\beta} + \frac{\beta
\norm{B}^{1+\beta}_{0,T,\nu}}{1+\beta}.
$$
Since $\norm{B}_{0,T,\nu}$ is an almost surely finite supremum of a centered  Gaussian family, and $\alpha(1+\beta)<2$, for any $y>0$ we have $\E{\exp\set{y\norm{B}_{0,T,\nu}^{\alpha(1+\beta)}}}<\infty$ thanks to Fernique's theorem. Further, it follows from Lemma~\ref{intlemma} that for any $y>0$ \ $\E{\exp\set{yJ_{X,\kappa}(T)^{\alpha(1+\beta)}}}<\infty$ and $\E{\exp\set{yJ_{X,\kappa}(T)^{\alpha}}}<\infty$. Thus, writing
$$
\norm{X}_{0,T,\infty}^\alpha\le C_\alpha \left(1+ \norm{B}_{0,T,\nu}^{\alpha(1+\beta)} + J_{X,\kappa}(T)^\alpha  +J_{X,\kappa}(T)^{\alpha(1+\beta)} \right),
$$
we get the required statement with the help of the  H\"older inequality.
\end{proof}

\begin{theorem}
Let $X$ be the solution of \eqref{eq2}. Then for all $t>0$ \ $X_t\in \bigcap_{p\ge 1}\mathbb D^{1,p}$.
\end{theorem}
\begin{proof}
Consider the sequence $\set{Z^n_t=n\int_{(t-1/n)\vee 0}^{t}B_s\d s,\, n\geq 1}$
of processes approximating $B$. It can be easily checked (see e.g. \cite{LRDShMish})  that $\norm{Z^n-B}_{0,T,\mu}\to 0$, $n\to \infty$ a.s. 
Processes $Z^n$ are absolutely continuous: $Z_t^n = \int_0^t \dot Z_s^n \d s$ with $\dot Z_t^n = n \big(B_{t} - B_{(t-1/n)\vee 0}\big)$.

Now define ${X}^n=\{X^n_t,t\in[0,T]\}=\{(X^{n,1}_t,\ldots,X^{n,d}_t),t\in [0,T]\}_{n\geq 1}$ as the solution to the SDE
 \begin{gather}
{X^n_t}=X_0 + \int_0^t \left(a(X^n_u)+c(X^n_u)\dot{Z^n_u}\right)\d u+\int_0^t b(X^n_u)\d W_u;
\label{eqXn}
\end{gather}
coordinatewise, for $i={1,\dots, d}$
\begin{gather*}
X^{n,i}_t=X_{0,i} + \int_0^t \Big(a_i(X^n_u)+\sum_{j=1}^l c_{i,j}(X^n_u)\dot Z^{n,j}_u\Big)\d u +\sum_{k=1}^m \int_0^t b_{i,k}(X^n_u)\d W^k_u,
\end{gather*}
where $\dot Z_t^{n,j} = n \big(B^{H,j}_{t} - B^{H,j}_{(t-1/n)\vee 0}\big)$.
By \cite[Theorem 4.1]{LRDShMish}, $X^n_t \to X_t$, $n\to \infty$, uniformly on $[0,T]$ in probability.

First note that all moments of $X^n$ are bounded uniformly in $n$. Indeed, from the almost sure convergence $\norm{Z^n-B}_{0,T,\mu}\to 0$, $n\to\infty$,  we have
$\zeta = \sup_{n\ge 1} \norm{Z^n}_{0,T,\mu}<\infty$ a.s. But $\zeta$ is a supremum of some centered Gaussian family
, so  by Fernique's theorem, $\E{\exp\set{{z \zeta^{a}}}}<\infty$ for all $z>0$, $a\in(0,2)$. A fortiori, $\sup_{n\ge 1}\E{\exp\set{{z \norm{Z^n}_{0,T,\mu}^{a}}}}<\infty$. Arguing as in the proof of Theorem~\ref{thm:integr}, we get $\sup_{n\ge 1}\E{\exp\set{z\norm{X^n}^\alpha_{0,T,\infty}}}<\infty$ for any $\alpha\in(0,4H/(2H+1))$, $z>0$. The uniform boundedness of moments clearly follows.

Further, the solution of (\ref{eqXn}) is Malliavin differentiable w.r.t.\ $W$ (see e.g. \cite{Nual}) as a solution to an It\^o SDE.  It is Malliavin differentiable w.r.t.\ $B$ as a solution of an It\^o SDE with a parameter.
To see this, take a direction $h\in \left(L^2_H [0,T]\right)^{ l}$ and consider a version of the equation \eqref{eqXn} with $B^H$ shifted by $\varepsilon h$, $\varepsilon\in \R$:
\begin{gather*}
X^{(n)}_t(\varepsilon,h)=x_0+\int_0^t n c(X^n_s(\varepsilon,h))\frac{\d{}}{\d{s}}\left(\int_{s-1/n}^s\left( B^H_u+\varepsilon\int_0^u \phi(u,v) h(v)\d{v}\right) \d{u}\right)\d{s}\\{}+
\int_0^t a(X^{(n)}_s(\varepsilon,h))\d{s}+\int_0^t b(X^{(n)}_s(\varepsilon,h))\d{W_s}.
\label{Expl}
\end{gather*}
This is a usual It\^o SDE with a scalar parameter $\varepsilon$, so its solution is differentiable with respect to $\varepsilon$; moreover, the derivative  $\frac{\d{}}{\d{\varepsilon}}X^{(n)}_t(\varepsilon,h)\bigr |_{\varepsilon=0}$ satisfies a linear SDE obtained by differentiating formally both sides of \eqref{Expl} (see e.g. \cite[Part II, \textsection 8]{GiSko}).  This means that $X^{(n)}_t$ is differentiable in all directions from $ \left(L^2_H [0,T]\right)^{l}$, as claimed.

The equations satisfied by the derivatives w.r.t.\ $B$ and $W$ are similar, so we will study those for $B$, as they are slightly more involved.

Fix $s\in [0,T]$ and $q=1,\dots,l$. Define $D^{n,i}_t= 
\md^{H,q}_s X^{i,n}_t$,  $i=1,\dots,d$.  Then $D^{n}$ satisfies
\begin{equation*} 
\begin{gathered}
D^{n,i}_t=D^{n,i}_{s,t}+\int_s^t\la \grad a_i(X^n_u),D^{n}_u\ra \d u +\sum_{j=1}^m \int_s^t \la \grad b_{i,j}(X^n_u), D^n_u\ra \d W^j_u 
\\+
\sum_{j=1}^l \int_s^t \la \grad c_{i,j}(X^n_u),D^n_u\ra \d Z^{j,n}_u,
\end{gathered}
\end{equation*}
where
$$D^{n,i}_{s,t}=\sum_{j=1}^l \int_0^t c_{i,j}(X^n_u) \md^{H,q}_s \dot Z^{n,j}_u \d u= n\int_{s}^{(s+1/n)\wedge t} c_{i,q}(X^n_u) \d u.$$
Due to linearity, the solution to equation~\ref{eqR} can be written as
$$
D_t^n = n\int_s^{(s+1/n)\wedge t} R^n_t(z) dz,
$$
where for $z\in (s,s+1/n)$ the process $\set{R^n_t(z) = \big(R^{n,1}_t(z),\dots,R^{n,d}_t(z)\big),t\ge z }$ solves
\begin{equation*} 
\begin{gathered}
R^{n,i}_t(z)= c_{i,q}(X_z)+\int_z^t\la \grad a_i(X^n_u),R^{n}_u(z)\ra \d u +\sum_{j=1}^m \int_z^t \la \grad b_{i,j}(X^n_u), R^{n}_u(z) \ra \d W^j_u 
\\+
\sum_{j=1}^l \int_z^t \la \grad c_{i,j}(X^n_u),R^{n}_u(z) \ra \d Z^{j,n}_u,
\end{gathered}
\end{equation*}

Therefore,
$$
\norm{D^n}_{s,T,\infty} \le n\int_s^{(s+1/n)} \norm{R^n(z)}_{z,T,\infty} dz,
$$
whence for any $p\ge 1$ by Jensen's inequality,
$$
\E{\norm{D^n}_{s,T,\infty}^p} \le n\int_s^{(s+1/n)} \E{\norm{R^n(z)}^p_{z,T,\infty}} dz.
$$
From Lemma~\ref{LemmaR} we have
$$
\E{\norm{R^n(z)}^p_{z,T,\infty}} \le K_p\left(\E{\exp\set{{K_p \norm{Z^n}_{0,T,\mu}^{\alpha}}}}\right)^{1/4}.
$$
As it was shown above, $\sup_{n\ge 1} \E{\exp\set{{K_p \norm{Z^n}_{0,T,\mu}^{\alpha}}}}<\infty$, thus, we obtain that $\E{\norm{D^n}_{s,T,\infty}^p}$ is bounded by a constant independent of $n$ and of $s$.
So we have for any $p\ge 1$,
$$
\sup_{n\ge 1}\sup_{s,t\in[0,T]}\E{\abs{\md^{H,q}_s X^n_t}^p}<\infty,\ q=1,\dots,l.
$$
Similarly, $$
\sup_{n\ge 1}\sup_{s,t\in[0,T]}\E{\abs{\md^{W,j}_s X^n_t}^p}<\infty,\ j=1,\dots,m. $$
Hence it is easy to deduce that
$\sup_{n\ge 1}\E{\norm{\md X^n_t}^p_{\mathfrak H}}<\infty$; and, taking into account that all moments of $X^n$ are bounded uniformly in $n$, we get $\sup_{n\ge 1}\E{\norm{X^n_t}^p_{\mathbb D^{1,p}}}<\infty$ for any $p\ge 1$. From here the Malliavin differentiability of $X$ follows from \cite[Lemma 1.2.3]{MR2200233} and uniform boundedness of moments $X^n$. 
\end{proof}
\begin{remark}
Using the same techniques, it is possible to generalize the results of the paper to the case where the driving fBm's have different Hurst exponents.
\end{remark}

\appendix
\section{Technical lemmas}\label{sec:bounds}
\begin{lemma}
\label{intlemma}
Let $A>0$, $\kappa\in(0,1/2)$, $\alpha\in(0,2)$, $z>0$, $t>0$.  There exists a constant $K_{A,\kappa,\alpha,z,t}$ such that if an $\mathbb F$-adapted process $\set{\xi_s,s\in [0,t]}$ satisfies $\abs{\xi_s}\le A$ for almost all $s\in[0,t]$ and $\omega\in \Omega$, and if $V$ is a scalar $\mathbb F$-Wiener process,  then  $\E{\exp\set{z\norm{\int_0^\cdot \xi_s dV_s}_{0,t,\kappa}^\alpha}}\le K_{A,\kappa,\alpha,z,t}.$
\end{lemma}
\begin{proof}
Define $Z_v=\int_0^v \xi_r\,d V_r$, $0\leq v\le t$ and put $Z_{u,s} =(Z_s-Z_u) (s-u)^{-\kappa},\ 0\le u<s\le t$. Using the Garsia--Rodemich--Rumsey inequality, we can write for $p>(1/2-\kappa)^{-1}$
\begin{gather*}
m:= \E{\norm{Z}_{0,t,\kappa}}\le C_{p,\kappa,t}
\E{\left(\int_0^t \int_0^t \frac{\abs{\int_u^s \xi_r dV_r}^{p}}{\abs{s-u}^{p\kappa + 2}}\d u\d s\right)^{1/p}}\le C_{p,\kappa,t}
\left(\int_0^t \int_0^t \frac{\E{\abs{\int_u^s \xi_r dV_r}^{p}}}{\abs{s-u}^{p\kappa + 2}}\d u\d s\right)^{1/p}\\
\le C_{p,\kappa,t} \left(\int_0^t \int_0^t \frac{ \Ex{\abs{\int_u^s \abs{\xi_r}^2 dr}^{p/2}}}{\abs{s-u}^{p\kappa + 2}}\d u\d s\right)^{1/p}\le
C_{p,\kappa,t} A \left(\int_0^t \int_0^t \abs{s-u}^{p(1/2-\kappa) -2 }\d u\d s\right)^{1/p}\le C_{p,\kappa,t,A}.
\end{gather*}
Further, let $\md^V$ denote the Malliavin derivative with respect to $V$, and $Z_{u,s} = \int_u^s \xi_r dV_r (s-u)^{-\kappa}$, $(u,s)\in \mathbf{T}:= \set{(a,b)\mid 0\le a<b\le t}$.
Then
$$
\int_{0}^{t}\abs{\md^V_r Z_{u,s}}^2 \d r = (u-s)^{-2\kappa}\int_{u}^{s}\abs{\xi_r}^2 \d r \le A^2 (u-s)^{1-2\kappa} \le A^2 t^{1-2\kappa}
$$
almost surely. Therefore, it follows from \cite[Theorem 3.6]{viens-vizcarra} that for any $x>0$
$$
\Pr\left(\norm{\int_0^\cdot \xi_s dV_s}_{0,t,\kappa}>m + x\right) = \Pr\left(\sup_{(u,s)\in\mathbf T} Z_{u,s}>m + x\right)\le 4\exp\set{-\frac{x^2}{2A^2 t^{1-2\kappa}}},
$$
which provides the required statement.
\end{proof}

Further we estimate the solution of a slightly more general version of equation \eqref{eq2}:
\begin{gather}
Y^i_t=Y^i_0+\int_0^ ta_i(Y_s)\d s+\sum_{j=1}^m\int_0^tb_{i,j}(Y_s)\d W^j_s+\sum_{k=1}^l\int_0^tc_{i,k}(Y_s)\d \gamma^k_s, i=1,\dots,d,
\label{SDEgamma}
\end{gather}
where $\gamma=\set{(\gamma^1_t,\ldots,\gamma^l_t),t\ge 0}$ is a process in $\R^d$ with $\mu$-H\"older continuous paths, $\mu>1/2$; the integral $\int_0^t c_{i,k}(X_s)\d \gamma^k_s$ is understood in the  Young sense.

Fix some $\theta\in(1-\mu,1/2)$ and define
\begin{gather*}
J_{Y,\theta}(t)= 
\sum_{i=1}^d\sum_{j=1}^m 
\norm{\int_0^\cdot b_{i,j}(Y_s)\d W^j_s}_{0,t,\theta} 
.
\end{gather*}

The following result establishes pathwise estimates of the solution to \eqref{SDEgamma}, which are better  than those in \cite[Lemma 4.1]{GMMixed}, but require stronger assumptions. To prove it, we modify the approach of  \cite{HuNualart}.
\begin{lemma}\label{lemmaX}
The solution $Y$ of \eqref{SDEgamma} satisfies
\begin{equation*}\label{Yestim}
\norm{Y}_{0,t,\infty}\leq \abs{Y_0}+2\left(\norm{a}_\infty+J_{Y,\theta}(t)+K_{\theta,\mu}\norm{\gamma}_{0,t,\mu}\norm{c}_\infty\right)\left(t^{\theta}+
t\left(2K_{\theta,\mu}\norm{\gamma}_{0,t,\mu}\norm{c'}_\infty+1\right)^{(1-\theta)/\mu}\right).
\end{equation*}
\end{lemma}
\begin{proof}To simplify the text, we will use the notation of equation \eqref{eq1} for the integrals in \eqref{SDEgamma}.
Applying (\ref{Estim}), we can write for $u,s\in[0,t]$ such that $s\in(u,u+1]$
\begin{align*}
\abs{Y_s-Y_u}
&\le \abs{\int_u^s a(Y_v)\d v} + \abs{\int_{u}^{s} b(Y_v)\d W_v} + \abs{\int_{u}^{s} c(Y_v)\d \gamma_v} \\
&\leq \norm{a}_\infty (s-u)+J_{Y,\theta}(s-u)^\theta +K_{\theta,\mu}\norm{\gamma}_{0,t,\mu}\Bigl(\norm{c}_\infty(s-u)^{\mu}+\norm{c(Y)}_{u,s,\theta}(s-u)^{\theta+\mu} \Bigr)\\&\leq
\big(\norm{a}_\infty+J_{Y,\theta}(t)\big)(s-u)^\theta+
K_{\theta,\mu}\norm{\gamma}_{0,t,\mu}\left(\norm{c}_\infty(s-u)^{\mu}+
 \norm{c'}_\infty\norm{Y}_{u,s,\theta} (s-u)^{\theta+\mu}\right).
\end{align*}
It follows that
\begin{gather*}
\norm{Y}_{u,s,\theta}\leq \norm{a}_\infty+J_{Y,\theta}(t)+K_{\theta,\mu}\norm{\gamma}_{0,t,\mu}\norm{c}_\infty+K_{\theta,\mu}\norm{\gamma}_{0,t,\mu}\norm{c'}_\infty\norm{Y}_{u,s,\theta}(s-u)^{\mu}.
\end{gather*}
Now put $\Delta= \left(2K_{\theta,\mu}\norm{\gamma}_{0,t,\mu}\norm{c'}_\infty+1\right)^{-1/\mu}$. With this choice, for $(s-u)\le \Delta$
\begin{equation}\label{Ysttheta}
\norm{Y}_{u,s,\theta}\leq 2\left( \norm{a}_\infty + J_{Y,\theta}(t) +K_{\theta,\mu}\norm{\gamma}_{0,t,\mu}\norm{c}_\infty\right).
\end{equation}
Therefore, for $(s-u)\le \Delta$
\begin{gather*}
\norm{Y}_{u,s,\infty}\leq \abs{Y_u} +\norm{Y}_{u,s,\theta}(s-u)^\theta \le  \abs{Y_u}+2\left(\norm{a}_\infty+J_{Y,\theta}+K_{\theta,\mu}\norm{\gamma}_{0,T,\mu}\norm{c}_\infty\right)(s-u)^{\theta}.
\end{gather*}
If $\Delta\ge t$, we set $u=0$, $s=t$ and obtain
\begin{gather*}
\norm{Y}_{0,t,\infty} \leq \abs{Y_0}+2\left(\norm{a}_\infty+J_{Y,\theta}(t)+K_{\theta,\mu}\norm{\gamma}_{0,t,\mu}\norm{c}_\infty\right)t^{\theta},
\end{gather*}
as needed. In the case where $\Delta<t$, write
\begin{gather*}
\norm{Y}_{u,s,\infty}\leq \norm{Y}_{0,u,\infty}+2\left(\norm{a}_\infty+J_{Y,\theta}(t)+K_{\theta,\mu}\norm{\gamma}_{0,t,\mu}\norm{c}_\infty\right)\Delta^\theta.
\end{gather*}
Hence, dividing the interval $[0,t]$ into $[t/\Delta]+1$ subintervals of length at most $\Delta$, we obtain by induction
\begin{align*}
\norm{Y}_{0,t,\infty}&\leq \abs{Y_0}+2\left(\norm{a}_\infty+J_{Y,\theta}(t)+K_{\theta,\mu}\norm{\gamma}_{0,t,\mu}\norm{c}_\infty\right)\Delta^\theta\left(1+[t/\Delta]\right) \\
&\leq
\abs{Y_0}+2\left(\norm{a}_\infty+J_{Y,\theta}(t)+K_{\theta,\mu}\norm{\gamma}_{0,t,\mu}\norm{c}_\infty\right)\left(t^\theta+t\Delta^{\theta-1}\right),
\end{align*}
which implies the required statement.
\end{proof}

For a fixed $s\in[0,T]$ and a fixed almost surely bounded $\mathcal F_s$-measurable random vector $R_s = (R^{1}_s,\ldots,R^{d}_s)\in\R^d$, let  $\set{R_t,t\in[s,T]}=\set{(R^{1}_t,\ldots,R^{d}_t),t\in[s,T]}$ be the solution of
\begin{equation}\label{eqR}
\begin{aligned}
R^i_t = R_s^i +\int_s^t\la \grad a_i(X^n_u),R_u\ra \d u +\sum_{j=1}^m \int_s^t \la \grad b_{i,j}(X^n_u), R_u\ra \d W^j_u +
\sum_{k=1}^l \int_0^t \la \grad c_{i,k}(X^n_u),R_u\ra \d \gamma^k_u,
\end{aligned}
\end{equation}
$i=1,\dots,d$, $t\in[s,T]$. We will write equation \eqref{eqR}  shortly as
$$
R_t = R_s + \int_s^t a'(Y_u) R_u\d u  + \int_s^t [b'(Y_u),R_u]\d W_u+\int_s^t  [c'(Y_u),R_u]\d\gamma_u.
$$
Using the same methods as in \cite{SDEMiSH,LRDShMish}, it can be shown that this equation has a unique solution such that $\norm{R}_{s,T,\theta}<\infty$ a.s.
\begin{lemma}
\label{LemmaR}
For any $p>0$
 $$\E{\norm{R}_{s,T,\infty}^p}\leq
 K_p \left(\E{\exp\set{{K_p \norm{\gamma}^{2/(\mu+\theta)}_{0,T,\mu}}}}\right)^{1/4},$$
where  the constant $K_p$ depends only on $p$, $T$, $\theta$, $\mu$, $\operatorname{ess\,sup} \abs{R_s}$, $\norm{a}_\infty$, $\norm{a'}_\infty$,
$\norm{b}_\infty$, $\norm{b'}_\infty$, $\norm{c}_\infty$, $\norm{c'}_\infty$, $\norm{c''}_\infty$.
\end{lemma}
\begin{proof}
In this proof the symbol $C$ will denote a generic constant which depends only on the parameters mentioned in the statement.

\newcommand{\I}{\mathbb{I}}
Fix some $N\ge 1$, $M\ge 1$ and define for $t\in [s,T]$ \ $A_t = \set{\norm{\gamma}_{0,t,\mu}+ J_{Y,\theta}(t)\le N, \norm{R}_{s,t,\infty}\le M}$, $\I_t = \ind{A_t}$, $Z_t = \int_s^t [b'(Y_s),R_s]\d W_s$. Let also $\Delta = \min\set{\Delta_1,\Delta_2,\Delta_3}$, where
\begin{align*}
\Delta_1 &=   \left(9K_{\theta,\mu}\norm{c'}_\infty N+1\right)^{-1/\mu},\quad \Delta_2 = \left(9\norm{a'}_\infty+1\right)^{-1},\\
\Delta_3& =\left(9K_{\theta,\mu}\norm{c''}_\infty\big(\norm{a}_\infty
+K_{\theta,\mu}
\norm{c}_\infty +1\big)N^2+1\right)^{-1/(\mu+\theta)}.
\end{align*}

Take some $u,t\in[s,T]$ such that $u<t$ and estimate
\begin{align*}
\abs{\int_u^t a'(Y_r)R_r\d r} &\leq \norm{a'}_\infty \norm{R}_{u,t,\infty} (t-u),\\
\abs{\int_u^t [c'(Y_r),R_r]\d \gamma_r}&\leq K_{\theta,\mu}  \norm{\gamma}_{0,t,\mu}\left( \norm{[c'(Y),R]}_{u,t,\infty}(t-u)^{\mu}+
\norm{[c'(Y),R]}_{u,t,\theta}(t-u)^{\theta+\mu}\right) \\
&\leq
K_{\theta,\mu} \norm{\gamma}_{0,t,\mu}\Bigl( \norm{c'}_{\infty}\norm{R}_{u,t,\infty}(t-u)^{\mu}
\\ &\qquad+\big(\norm{R}_{u,t,\infty}\norm{c''}_\infty\norm{Y}_{u,t,\theta}+
\norm{c'}_\infty\norm{R}_{u,t,\theta}\big) (t-u)^{\theta+\mu}   \Bigr).
\end{align*}
Therefore,
\begin{align*}
&\abs{R_t - R_u}  \le K_{\theta,\mu}\norm{\gamma}_{0,t,\mu} \norm{c'}_\infty (t-u)^{\theta+\mu}\norm{R}_{u,t,\theta} + \abs{Z_t-Z_u}
\\
& + \left(\norm{a'}_\infty  (t-u) +  K_{\theta,\mu}\norm{\gamma}_{0,t,\mu}\big(\norm{c'}_{\infty}(t-u)^{\mu} +\norm{c''}_\infty\norm{Y}_{u,t,\theta}(t-u)^{\theta+ \mu}\big)
\right)\norm{R}_{u,t,\infty}.
\end{align*}
Now let $(t-u)\le \Delta$ and $\omega\in A_t$. Then
\begin{align*}
&\norm{R}_{u,t,\theta} \le K_{\theta,\mu} \norm{c'}_\infty N\Delta^{\mu}\norm{R}_{u,t,\theta} + \norm{Z}_{u,t,\theta}
\\
&\quad + \left(\norm{a'}_\infty  \Delta^{1-\theta} +  K_{\theta,\mu}N\big(\norm{c'}_{\infty}\Delta^{\mu-\theta} +\norm{c''}_\infty\norm{Y}_{u,t,\theta}\Delta^{\mu}\big)
\right)\norm{R}_{u,t,\infty}\\
&\le \frac19 \norm{R}_{u,t,\theta} + \norm{Z}_{u,t,\theta} + \left(\frac{1}{9} \Delta^{-\theta} +  \frac19\Delta^{-\theta} + K_{\theta,\mu} \norm{c''}_\infty \norm{Y}_{u,t,\theta} N \Delta^{\mu}\big)
\right)\norm{R}_{u,t,\infty}. 
\end{align*}
Applying the estimate \eqref{Ysttheta} and the definition of $\Delta_3$, we get $K_{\theta,\mu} \norm{c''}_\infty \norm{Y}_{u,t,\theta} N\Delta^{\mu+\theta}\le 2/9$, which leads to
$$
\norm{R}_{u,t,\theta}\le \frac1{2 \Delta^{\theta}} \norm{R}_{u,t,\infty} + \frac{9}{8}\norm{Z}_{u,t,\theta}.
$$
Obviously, $\norm{R}_{u,t,\infty} \le \abs{R_u} + \norm{R}_{u,t,\theta}\Delta^\theta \le \norm{R}_{s,u,\infty} +  \norm{R}_{u,t,\infty}/2 + 9\Delta^\theta \norm{Z}_{u,t,\theta}/8$, whence
$$
\norm{R}_{s,t,\infty} \le 2 \norm{R}_{s,u,\infty} + \frac{9\Delta^\theta}{4}\norm{Z}_{u,t,\theta}
$$
for $\omega\in A_t$.
Therefore,
\begin{equation}\label{normr2p}
\E{\norm{R}^{2p}_{s,t,\infty}\I_t}\le C \left(\E{\norm{R}^{2p}_{s,u,\infty}\I_t} +  \Delta^{2p\theta}\E{\norm{Z}^{2p}_{u,t,\theta}\I_t} \right) \le
C \left(\E{\norm{R}^{2p}_{s,u,\infty}\I_s} +  \E{\norm{Z}^{2p}_{u,t,\theta}\I_t}\right).
\end{equation}
We can assume without loss of generality that $p>2(1-2\theta)^{-1}$. Then, by the Garsia--Rodemich--Rumsey inequality,
\begin{gather*}
\E{\norm{Z}^{2p}_{u,t,\theta}\I_t}\le C \int_u^t \int_u^t \frac{\E{\abs{Z_x-Z_y}^{2p}\I_t}}	{\abs{x-y}^{2\theta p + 2}}\d x\d y
\le C \int_u^t \int_u^t \frac{\E{\abs{\int_x^y [b'(Y_r),R_r]\I_r\d W_r}^{2p}}}{\abs{x-y}^{2\theta p + 2}}\d x\d y\\
\le C \int_u^t \int_u^t \frac{\E{\abs{\int_x^y \norm{b'}_{\infty}^2 \abs{R_r}^2\I_r\d r}^{p}}}{\abs{x-y}^{2\theta p + 2}}\d x\d y\le
C \int_u^t \int_u^t \int_x^y \E{\abs{R_r}^{2p}\I_r}\d r\abs{x-y}^{p(1-2\theta)  - 3}\d x\d y\\
\le C \int_{u}^{t}\E{\abs{R_r}^{2p}\I_r}\d r \int_0^t \int_0^t\abs{x-y}^{p(1-2\theta)  - 3}\d x\d y \le C\int_{u}^{t}\E{\norm{R}^{2p}_{s,r,\infty}\I_r}\d r.
\end{gather*}
Plugging this estimate into  \eqref{normr2p} we get with the help of the Gronwall lemma that
\begin{gather*}
\E{\norm{R}^{2p}_{s,t,\infty}\I_t} \le C \E{\norm{R}^{2p}_{s,u,\infty}\I_u} e^{C\Delta}\le C \E{\norm{R}^{2p}_{s,u,\infty}\I_u}
\end{gather*}
whenever $(t-u)<\Delta$. By induction, we get
\begin{equation}\label{normr2p1}
\E{\norm{R}^{2p}_{s,T,\infty}\I_T} \le \E{\abs{R_s}^{2p}\I_s}C^{[(T-s)/\Delta]+1} \le C e^{CT/\Delta}
\le K e^{K N^\alpha},
\end{equation}
where we use the symbol $K$ for the constant, as it will be fixed from now on. Denote $\zeta = \norm{\gamma}_{0,T,\mu}+ J_{Y,\theta}(T)$ and $\alpha = 2/(\mu+\theta)$. Observe that in \eqref{normr2p1} the right-hand side is independent of $M$, so letting $M\to\infty$, we get $\E{\norm{R}^{2p}_{s,T,\infty}\ind{\zeta\le N}}  \le K \exp\set{K N^\alpha}$ for any $N\ge 1$. Now write
\begin{gather*}
\E{\norm{R}^{2p}_{s,T,\infty}e^{-2 K \zeta^\alpha}}
= \sum_{n=1}^{\infty} \E{\norm{R}^{2p}_{s,T,\infty}e^{-2 K \zeta^\alpha} \ind{\zeta\in[n-1,n]}}\\
\le \sum_{n=1}^{\infty} \E{\norm{R}^{2p}_{s,T,\infty}e^{-2 K (n-1)^\alpha} \ind{\zeta\le n}}\le K \sum_{n=1}^{\infty} e^{Kn^\alpha-2 K (n-1)^\alpha}=:K'<\infty.
\end{gather*}
By the Cauchy--Schwartz inequality,
\begin{gather*}
\E{\norm{R}^{p}_{s,T,\infty}}\le \left(\E{\norm{R}^{2p}_{s,T,\infty}e^{-2 K \zeta^\alpha}} \E{e^{2 K \zeta^\alpha}}\right)^{1/2}
\le \left(K' \E{e^{2 K \zeta^\alpha}}\right)^{1/2}.
\end{gather*}
Finally,
\begin{gather*}
\E{e^{2 K \zeta^\alpha}}\le \E{\exp\set{2^\alpha K \big(\norm{\gamma}_{0,T,\mu}^\alpha + J_{Y,\theta}(T)^\alpha\big)}}\\
\le \left(\E{\exp\set{2^{\alpha+1} K\norm{\gamma}_{0,T,\mu}^\alpha}} \E{\exp\set{2^{\alpha+1} K J_{Y,\theta}(T)^\alpha\big)}}\right)^{1/2}.
\end{gather*}
Applying Lemma~\ref{intlemma} to the last term, we get the required statement.
\end{proof}
\textbf{Acknowledgments.} The authors would like to thank the anonymous referees of the manuscript and the area editor for their valuable comments and suggestions that helped to improve the quality of the paper.
\bibliographystyle{elsarticle-harv}
\bibliography{malliavinregularity}
\end{document}